\documentclass[12pt]{article}
\usepackage{amssymb,latexsym,amsmath,epsfig,amsthm} %% Add other packages as necessary

\newtheorem{theorem}{Theorem}
\newtheorem{lemma}{Lemma}

\newtheorem{proposition}{Proposition}

\def\OK{\mathcal{O}_K}
\def\Q{\mathbb{Q}}

\begin{document}

%\title{Counting Eisenstein polynomials satisfying a condition from genus theory}
%\author{Jongwoo Choi, Kevin McGown}

\begin{center}
\uppercase{\bf Counting Eisenstein polynomials satisfying a condition from genus theory}
\vskip 20pt
{\bf Jongwoo Choi and Kevin J. McGown}\\
June, 2020
\end{center}

\centerline{\bf Abstract}
\noindent
We give an asymptotic formula for the number of monic Eisenstein polynomials of odd prime degree satisfying
an additional condition that arises in the study of the genus number of an algebraic number field.

\section{Introduction}

Let $d$ denote an odd prime throughout.
Our main theorem concerns counting the proportion of
Eisenstein polynomials of degree $d$ satisfying a
specific condition.
Before commencing with the statement of our result,
it is perhaps useful to comment on where this condition arises.

%Unless otherwise stated, $d$ will always denote an odd prime.
Let $K$ be an algebraic number field of degree $d$.
One important invariant associated to $K$ is the so-called
genus number of $K$, denoted $g_K$;
in our setting $g_K$ is a divisor of the class number $h_K$.
It is very natural to ask about
the statistical distribution of $g_K$ as one runs through all number fields
of a given degree.  More specifically, one could ask what proportion
of such fields have $g_K=1$.   Asymptotically, cyclic fields constitute zero percent
of all fields of degree $d$, so we may assume our field $K$ to be non-cyclic.

In~\cite{MT}, A.~Tucker and  the second author establish the precise proportion of
cubic fields with genus number one; establishing a similar result for $d=5$
is work in progress.
However, when $d>5$, counting number fields of degree $d$
is an open problem.
%In the interest of doing something when $d>5$, we will instead focus
%on the polynomials that generate the fields.

It is well-known that
one can choose $\alpha\in\OK$ with $K=\Q(\alpha)$
such that the minimal polynomial $f(x)$ of $\alpha$ is $p$-Eisenstein
iff $p$ is totally ramified in $K$.
%(This statement still holds even when there are no totally ramified primes.)
For example, see Ch.~2 of~\cite{Ishida} for the precise recipe.
Provided the generating polynomial $f(x)$ is chosen in this way,
a theorem of Ishida gives a precise method for determining the genus number of $K$;
in particular, this allows one to decide whether $K$ has genus number one
just by looking at $p$-divisibility of the polynomial.

Write
\[
f(x)=x^d+a_{d-1}x^{d-1}+\dots +a_1 x+a_0
\]
and let $P$ denote the set of primes $p$ for which $f(x)$ is $p$-Eisenstein.
Then $g_K>1$ if and only if $p\equiv 1\pmod{d}$ for some $p\in P$
\emph{or}
\[
  d\in P \;\text{ and }\; a_2\equiv\dots\equiv a_{d-1}\equiv a_1+a_d\equiv 0\pmod{d^2}
  \,.
\]
Putting aside the question of counting fields,
one could ask, what proportion
of Eisenstein polynomials $f(x)$ of degree $d$ fail to satisfy this condition?
For brevity, let us denote the negation of this condition by $(\star)$
so that, in this context, $g_K=1$ iff condition $(\star)$ holds.

Let $\mathcal{E}_d(H)$ denote the collection of all monic Eisenstein polynomials
of height at most $H$ and
let $\mathcal{E}_d^{\star} (H)$ denote the collection of all
$f \in \mathcal{E}_d(H)$
such that $f$ satisfies condition $(\star)$.
Building on work of Dubickas~\cite{D}, Heymann and Shparlinski~\cite{HS} prove that
\begin{equation}\label{E:HS}
  \#\mathcal{E}_d(H)=\theta_d (2H)^d+
  \begin{cases}
   O(H^{d-1}) &\text{ if $d>2$ }\\
   O(H(\log H)^2) &\text{ if $d=2$}\\ 
  \end{cases}
\end{equation}
where
\[
  \theta_d=1-\prod_p
  \left(1-\frac{p-1}{p^{d+1}}\right)
  \,.
\]
Along the same lines,
we obtain the following  asymptotic:

\begin{theorem}\label{T:main}
\[
  \#\mathcal{E}_d^{\star} (H)=\theta^{\star}_d(2H)^d +
  \begin{cases}
   O(H^{d-1}) & \text{ if $d>2$ }\\
   O(H(\log H)^2) & \text{ if $d=2$ }\\ 
  \end{cases}
\]
where
\[
  \theta^{\star}_d = 1-\frac{d-1}{d^{2d}}-\left(1-\frac{(d-1)(d^{d-1}+1)}{d^{2d}}\right) \prod_{\substack{p \neq d\\ p \not\equiv 1 (d)}} \left(1-\frac{p-1}{p^{d+1}}\right)
  \,.
\]
In particular,
\[
  \lim_{H\to\infty}\frac{ \#\mathcal{E}_d^{\star} (H)}{\#\mathcal{E}_d(H)}=\frac{\theta^{\star}_d}{\theta_d}
  \,.
\]
\end{theorem}

As a consequence, the answer to our question is given explicitly by the ratio $\theta_d^\star/\theta_d$.
When $d=3$, we have $\theta^\star_3/\theta_3\approx 0.9681192$, and the sequence of these ratios
tend to $1$ as $d\to\infty$.
%increases to one.
We emphasize that counting the polynomials is different than counting the fields,
but we found determining a closed form for the ratio
$\theta^{\star}_d/\theta_d $
an interesting problem nonetheless.

\section{Preliminary Lemmas}

We define
\[
  \varphi(s,H)=\sum_{\substack{|a|\leq H\\\gcd(a,s)=1}}1
  \,.
\]
\begin{lemma}\label{L:Dubickas}
For any integer $s\geq 1$, we have
\[
  \varphi(s,H)=\frac{2H\varphi(s)}{s}+O\left(2^{\omega(s)}\right)
  \,.
\]
\end{lemma}

\begin{proof}
See Lemma~4 of~\cite{HS}.
\end{proof}

Let $\mathcal{G}_d(s,H)$ be the set of degree $d$ monic polynomials of height at most $H$ satisfying:
\begin{enumerate}
\item
$s\mid a_i$ for $i=0,\dots,d-1$,
\item
$\gcd(a_0/s,s)=1$.
\end{enumerate}

\begin{lemma}\label{L:G}
For $s\leq H$, we have
\[
  \#\mathcal{G}_d(s,H) = \frac{2^{d}H^{d}\varphi(s)}{s^{d+1}} + O\left(\frac{H^{d-1} 2^{\omega(s)}}{s^{d-1}}\right)
  .
\]
\end{lemma}

\begin{proof}
See Lemma~3 of~\cite{HS}.
\end{proof}

%\begin{proof}
%Assume $s \le H$. For every $i = 1, 2, ... , d-1$, the number of possibilities for each $a_i$ is equal to 
%\[
%  2 \left\lfloor\frac{H}{s}\right\rfloor + 1 = \frac{2H}{s} + O\left(1\right).
%\]
%
%We can take $a_0 = sm$ with an integer $m$ satisfying $|m| \le H/s$ and \mbox{$\gcd(m,s)=1$.} 
%Using Lemma \ref{L:Dubickas}, the number of possibilities for $a_0$ is
%\[
%  \varphi\left(s, \frac{H}{s}\right) = \frac{2H\varphi(s)}{s^2} + O\left(2^{\omega(s)}\right). 
%\]
%
%Therefore,
%\begin{align*}
%  \#\mathcal{G}_d(s,H) =& \left(\frac{2H}{s} + O\left(1\right)\right)^{d-1} \left(\frac{2H\varphi(s)}{s^2} + O\left(2^{\omega(s)}\right)\right) \\ 
%                                  =& \left(\left(\frac{2H}{s}\right)^{d-1} + O\left(\left(\frac{H}{s}\right)^{d-2}\right)\right)  \left(\frac{2H\varphi(s)}{s^2} + O\left(2^{\omega(s)}\right)\right) \\
%						  =& \frac{2^{d}H^{d}\varphi(s)}{s^{d+1}} + O\left(\frac{H^{d-1} 2^{\omega(s)}}{s^{d-1}}\right)
%						  .
%\end{align*}
%
%\end{proof}

Let $\mathcal{G}'_d(s,H)$ be the set of degree $d$ monic polynomials $f(x)$ of height at most $H$ satisfying:
\begin{enumerate}
\item
$s\mid a_i$ for $i=0,\dots,d-1$,
\item
$\gcd(a_0/s,s)=1$,
\item
$f(x)$ is Eisenstein at $d$,
\item
$ a_1\equiv\dots\equiv a_{d-2}\equiv a_0+a_{d-1}\equiv 0\pmod{d^2}$.
\end{enumerate}

\begin{lemma}\label{L:G1}
For $s\leq H$ with $\gcd(s,d)=1$, we have
\[
  \#\mathcal{G}'_d(s,H) = \frac{2^{d}H^{d}\varphi(ds)}{s^{d+1}d^{2d}} + O\left(\frac{H^{d-1} 2^{\omega(s)}}{s^{d-1}}\right)
  .
\]
\end{lemma}

\begin{proof}
Assume $s \le H$. For every $i = 1,2, ..., d-2$, we have $d^2 s \mid a_i$ and thus the number of possibilities for each $a_i$ is equal to
\[
2 \left\lfloor \frac{H}{d^2s}\right\rfloor +1.
\]
We then wish to count the number of integers $|a_0| \le H$ satisfying $ds \mid a_0$ and
$\gcd(a_0/(ds), ds) = 1$. Since $ds\mid a_0$, we may write $a_0 = kds$. Using Lemma \ref{L:Dubickas}, the number of possible $a_0$ is equal to
\[
\varphi\left(ds, \frac{H}{ds}\right) = 2 \frac{\varphi(ds)}{ds} \frac{H}{ds} + O\left(2^{\omega(ds)}\right) = 2H\frac{\varphi(ds)}{d^{2}s^{2}} + O\left(2^{\omega(s)}\right) .
\]
Having chosen $a_0$, we want to count $|a_{d-1}| \le H$ satisfying $d\mid a_{d-1}$, $s\mid a_{d-1}$, and  $a_0 + a_{d-1} \equiv 0 \pmod {d^2}$. From the last condition, we know that  $d$ divides $a_0 + a_{d-1}$, and therefore $d \mid a_{d-1}$.
Hence we may drop the first condition. 
Thus, the number of possibilities for $a_{d-1}$ is
\[
%O\left(\frac{H}{d^2s} \right) .
\frac{2H}{d^2 s}+O(1).
\]
Therefore, 
\begin{align*}
 \#\mathcal{G}'_d(s,H)
 &=
 \left(\frac{2H}{d^{2}s} + O(1) \right)^{d-1} \left(2H \frac{\varphi(ds)}{d^{2}s^{2}} + O\left(2^{\omega(s)}\right) \right)   \\
&=
\left(
\left(\frac{2H}{d^2 s}\right)^{d-1}+O\left(\left(\frac{H}{s}\right)^{d-2}\right)
\right)
\left(
2H \frac{\varphi(ds)}{d^{2}s^{2}} + O\left(2^{\omega(s)}\right)
%\left(\frac{2^{d-2}H^{d-2}}{s^{d-2}d^{2d-4}} + O \left( \left(\frac{H}{s}\right)^{d-2} \right) \right) \cdot \left(\frac{H}{d^{2}s} + O \left(\frac{H}{s}\right) \right) \cdot \left(2H \frac{\varphi(ds)}{d^{2}s^{2}} + O(2^{\omega(s)})
\right) \\
&=
\frac{2^{d}H^{d}\varphi(ds)}{s^{d+1}d^{2d}} + O\left(\frac{H^{d-1} 2^{\omega(s)}}{s^{d-1}}\right)
.
\end{align*}
\end{proof}

\section{Proof of Theorem \ref{T:main}}

Given $f \in \mathcal{E}_d(H)$, suppose $f$ is Eisenstein at $p_1, p_2, ..., p_t$ and no other primes except possibly $d$.
We consider the following two sets:
\begin{align*}
 \mathcal{E}_d^{(1)}(H) &= \{f \in \mathcal{E}_d(H) : p_i \equiv 1\pmod{d}\text{ for some $i$}\}\,, \\
 \mathcal{E}_d^{(2)}(H) &= \{f \in \mathcal{E}_d(H) : \text{$f$ is $d$-Eisenstein, $p_i \not\equiv 1\pmod{d}$ for $i = 1,2, ..., t$}, \\
							 &\qquad \text{and }a_1\equiv\dots\equiv a_{d-2}\equiv a_0+a_{d-1}\equiv 0\pmod{d^2}\}\,.
\end{align*}

We observe that $\#\mathcal{E}_d^{\star}(H) =\#\mathcal{E}_d(H) - \#\mathcal{E}_d^{(1)}(H) - \#\mathcal{E}_d^{(2)}(H)$.

\begin{proposition}\label{P:main} 
We have
\[
  \#\mathcal{E}_d^{(1)} (H)=\alpha_d(2H)^d +
  \begin{cases}
   O(H^{d-1}) \text{ if $d>2$ }\\
   O(H(\log H)^2) \text{ if $d=2$ }\\ 
  \end{cases}
\]
where
\[
  \alpha_d = - \left(1 - \frac{d-1}{d^{d+1}} \right)  \prod_{\substack{p \not\equiv 1 (d)\\p \neq d}} \left( 1 - \frac{p-1}{p^{d+1}} \right)  \left[ -1+\prod_{p \equiv 1 (d)} \left(1- \frac{p-1}{p^{d+1}} \right) \right] .
  \]
\end{proposition}

\begin{proof}
The idea here is to mimic the proof of Theorem~1 from \cite{HS} but with the extra
conditions we need thrown in.
Let $\mathcal{S}$ be the set of square-free positive integers divisible by at least one prime $p \equiv 1 \pmod{d}$.
%In \cite{HS}, they write:
%\[
%  \#\mathcal{E}_d(H)=-\sum_{s=2}^H \mu(s)\#\mathcal{G}_d(s,H)
%\]
Following~\cite{HS} and applying Lemma~\ref{L:G}, we have
\begin{align*}
   \#\mathcal{E}_d^{(1)}(H)
   &=
   -\sum_{\substack{s=2\\s \in \mathcal{S}}}^H
   \mu(s)\#\mathcal{G}_d(s,H)
   \\
  &=
  - \sum_{\substack{s=2\\ s \in \mathcal{S}}}^H
  \mu(s) \left(\frac{2^{d} H^{d} \varphi(s)}{s^{d+1}}\right) + O\left(\sum_{s=2}^H \left(\frac{H}{s}\right)^{d-1} 2^{\omega(s)} \right)
  \\ 
  &=
  - \sum_{\substack{s=2\\s \in \mathcal{S}}}^\infty (2H)^d \frac{\mu(s)\varphi(s)}{s^{d+1}} + O\left(H^d \sum_{s=H+1}^\infty \frac{\varphi(s)}{s^{d+1}} +   \sum_{s=2}^H \left(\frac{H}{s}\right)^{d-1} 2^{\omega(s)} \right) 
\end{align*}
Since $\phi(s) \le s$, one has
\[
H^d \sum_{s=H+1}^\infty \frac{\varphi(s)}{s^{d+1}}= O(H). 
\]
Moreover, one has
\[
\sum_{s=1}^H \frac{2^{\omega(s)}}{s^{d-1}} = 
\begin{cases}
   O(1) & \text{ if $d>2$ }\\
   O((\log H)^2) & \text{ if $d=2$ }\\ 
  \end{cases}.\]
  See equations (11) and (12) in the proof of Theorem~1 from \cite{HS}.
We also have 
\[
- \sum_{\substack{s=2\\ s \in \mathcal{S}}}^\infty
\frac{\mu(s)\varphi(s)}{s^{d+1}} =   - \left(1 - \frac{d-1}{d^{d+1}} \right)
\prod_{\substack{p \not\equiv 1 (d)\\ p \neq d}} \left( 1 - \frac{p-1}{p^{d+1}} \right)  \left[ -1+\prod_{p \equiv 1 (d)} \left(1- \frac{p-1}{p^{d+1}} \right) \right]
.
\]
This completes the proof.

\end{proof}

\begin{proposition}\label{P:main2}
We have
\[
  \#\mathcal{E}_d^{(2)} (H)=\beta_d(2H)^d+
  \begin{cases}
   O(H^{d-1}) & \text{ if $d>2$ }\\
   O(H(\log H)^2) & \text{ if $d=2$ }\\ 
  \end{cases}
\]
where
\[
  \beta_d= 
  \frac{d-1}{d^{2d}}
  \left(
  1-
  \prod_{\substack{p \not\equiv 1 (d)\\p \neq d}}\left(1-\frac{p-1}{p^{d+1}}
  \right)
  \right) .
\]
\end{proposition}

\begin{proof}
Let $\mathcal{S}'$ be the set of square-free positive integers coprime to $d$ which are products of primes $p \not\equiv 1 \pmod d$.
Applying Lemma~\ref{L:G1}, we have:
\begin{align*}
  \#\mathcal{E}_d^{(2)} (H) &= - \sum_{\substack{s=2\\ s \in \mathcal{S'}}}^H \mu(s)\#\mathcal{G}'_d(s,H) \\
		       &= - \sum_{\substack{s=2\\ s \in \mathcal{S'}}}^H \mu(s)(2H)^d \frac{\varphi(ds)}{s^{d+1}d^{2d}} + O\left(\sum_{s=2}^H \left(\frac{H}{s}\right)^{d-1} 2^{\omega(s)} \right)  \\
		       &= - (2H)^d \frac{d-1}{d^{2d}}  \sum_{\substack{s=2\\ s \in \mathcal{S'}}}^{\infty}  \frac{\mu(s)\varphi(s)}{s^{d+1}} +  O\left(H^d \sum_{s=H+1}^\infty \frac{\varphi(s)}{s^{d+1}} +   \sum_{s=2}^H \left(\frac{H}{s}\right)^{d-1} 2^{\omega(s)} \right). \\
\end{align*}
The error terms are handled as in the proof of the previous proposition.
Finally, to complete the proof, we observe that
\[
-  \sum_{\substack{s=2\\ s \in \mathcal{S'}}}^\infty  \frac{\mu(s)\varphi(s)}{s^{d+1}} = 
1-\prod_{\substack{p \not\equiv 1 (d)\\ p \not = d}}\left(1-\frac{p-1}{p^{d+1}}\right)
.
\]

\end{proof}

Putting together Proposition~\ref{P:main}, Proposition~\ref{P:main2}, and~(\ref{E:HS}) establishes Theorem~\ref{T:main} in light of the fact that $\theta_d^\star=\theta_d-\alpha_d-\beta_d$.

\section*{Acknowledgement}

This work represents a portion of the first author's honors thesis
at California State University, Chico under the supervision of the second author.
The authors would like to thank the Mathematics and Statistics Department 
for their support.  In addition, the authors thank the anonymous referee for their
helpful suggestions.

%J. Deer and K. Doe, On the history of mathematics, {\it J. of the World} {\bf 52} (1999), 123-135.
%A. Jones, L. Smith, and C. Vector, {\it The Theory of Everything}, Publishing Company, New York, 1987.

\end{document}